\documentclass[11pt]{article}
\usepackage{amsfonts,amsmath,amssymb,amstext,amsbsy,amsopn,amsthm}
\usepackage{mathptmx}
\usepackage{url}
\usepackage{graphicx}
\usepackage{amsbsy}
\usepackage{booktabs}

\newcommand\R{{\mathbb R}}
\newcommand\Z{{\mathbb Z}}

\newcommand{\itg}{\int \limits}

\newcommand{\ee}{{\rm e}}

\newcommand\grad{{\rm grad}\,}

\newcommand\erf{{\rm erf}\,}
\newcommand\ddiv{{\rm div}\,}

\newcommand\e{\varepsilon}
\newcommand\eps{\varepsilon}

\newcommand\h{\eta}

\newcommand\x{\xi}

\newcommand{\cD}{\mathcal D}

\newcommand{\cH}{\mathcal H}

\newcommand{\cL}{\mathcal L}
\newcommand{\cM}{\mathcal M}
\newcommand{\cN}{\mathcal N}

\newcommand{\cO}{\mathcal O}
\newcommand{\cP}{\mathcal P}
\newcommand{\cQ}{\mathcal Q}
\newcommand{\cR}{\mathcal R}

\newcommand\bx{{\mathbf x}}
\newcommand\by{{\mathbf y}}
\newcommand\bm{{\mathbf m}}
\newcommand\bk{{\mathbf k}}
\newcommand\bu{{\mathbf u}}
\newcommand\ff{{\mathbf f}}
\newcommand\bg{{\mathbf g}}

\newtheorem{thm}{Theorem}[section]

\title{Approximation of solutions to non-stationary Stokes system}

\begin{document}

\author{ Flavia Lanzara \thanks{Department of Mathematics, 
Sapienza University of Rome, 
Piazzale Aldo Moro 2, 00185 Rome, Italy. 
{\it email:} \texttt{\rm lanzara\symbol{'100}mat.uniroma1.it}}
\and Vladimir Maz'ya \thanks{Department of Mathematics, University of
Link\"oping,  581 83 Link\"oping, Sweden.} \thanks{RUDN University, 6 Miklukho-Maklay St, Moscow, 117198, Russia. 
{\it email:} \texttt{\rm vlmaz\symbol{'100}mai.liu.se}}
\and Gunther Schmidt \thanks{Berlin, Germany.}
}

\date{}
\maketitle
 
{\bf Abstract:} We propose a fast method for high order approximations of the solution of the
Cauchy problem for the linear non-stationary Stokes system in \(\R^3\) in the unknown velocity \(\bu\) and  kinematic pressure \(P\). The density \(\bf f(\bx,t)\) and the divergence vector-free initial value \({\bg }(\bx)\) are smooth and rapidly decreasing as \(|\bx|\) tends to infinity. 
We construct the vector \(\bu\) in the form \(\bu=\bu_1+\bu_2\) where \(\bu_1\) solves a system of homogeneous heat equations and \(\bu_2\) solves a system of non-homogeneous heat equations with right-hand side \({\bf f}-\nabla P\).
Moreover  \(P=-\cL( \nabla\cdot \bf f)\) where \(\cL\) denotes the harmonic potential.
 Fast semi-analytic cubature formulas for computing the harmonic potential  and the solution of the heat equation based on the approximation of the data by functions with analitically known potentials are considered. In addition, the gradient \(\nabla P\) can be approximated by the gradient of the cubature of \(P\), which is a semi-analytic formula too. We derive fast and accurate high order formulas for the approximation of \(\bu_1,\bu_2\), \(P\) and \(\nabla P\). The accuracy of the method and the convergence order \(2,4,6\) and \(8\) are confirmed by numerical experiments.


\section{Introduction}

In the present paper we describe a fast method for the numerical solution of the non-stationary Stokes system
\begin{eqnarray}\label{eq1}
\frac{\partial \bu}{\partial t} -\nu \Delta \bu+\nabla  P=\ff \\\label{eq2}
\nabla\cdot \bu=0
\end{eqnarray}
with the initial condition
\begin{eqnarray}
\label{eq3}
\bu(\bx,0)=\bg (\bx) 
\end{eqnarray}
for \((\bx,t)\in\R^3\times\R_+\) with \(\R_+=[0,\infty)\). 
Here \(\nabla=\{\partial_{x_1},\partial_{x_2},\partial_{x_3}\}\) so that \(\nabla P=\grad P\) and \(\nabla\cdot \bu=\ddiv \bu\)\,.
Equations {\eqref{eq1}-\eqref{eq3}} are solved for an unknown velocity vector { field} \(\bu(\bx,t)=(u_1(\bx,t),u_2(\bx,t),u_3(\bx,t))\)  and kinematic pressure \(P(\bx,t)\), defined at all points \(\bx\in\R^3\) and time \(t\geq 0\).  Here the viscosity \(\nu\) is a positive coefficient,
the data \(\ff(\bx,t)=(f_1(\bx,t),f_2(\bx,t),f_3(\bx,t))\) and \(\bg(\bx)=(g_1(\bx),g_2(\bx),g_3(\bx))\) are smooth and sufficiently rapidly decreasing as \(|\bx|\to \infty\) and \(\bg\) is a divergence-free vector field, { i.e. \(\nabla\cdot \bg=0\).}
The solution of the Cauchy problem {\eqref{eq1}-\eqref{eq3}} for smooth functions \(\ff\) and \(\bg\) decreasing at infinity is of the form 
\[
\bu (\bx,t)=\bu_1 (\bx,t)+\bu_2 (\bx,t),
\]
where the vector \(\bu_1\) is the solution of the homogeneous heat equations
\begin{equation}\label{heat}
\partial_t \bu_1-\nu \Delta \bu_1=0,\quad
\bu_1(\bx,0)=\bg(\bx)
\end{equation}
and the { pair} \((\bu_2,P)\)  is the solution of the Cauchy problem
\begin{equation}\label{heat2}
\partial_t \bu_2-\nu \Delta \bu_2+\nabla P=\ff,\quad
\nabla\cdot \bu_2=0,\quad \bu_2(\bx,0)=0\,.
\end{equation}
By the unique solvability of the Cauchy problem for the heat equation, from the condition \(\nabla\cdot \bg=0\) it follows that \(\nabla\cdot \bu_1=0\) for all \(t>0\). 
The { pair} \((\bu_1,\bu_2)\) { and \(P\)} have the form
\begin{eqnarray}\notag
\bu_1(\bx,t)&=&(\cP \bg)(\bx,t):=\itg_{\R^3} \Gamma(\bx-\by,t)\bg(\by)d\by
\\\notag
\bu_2(\bx,t)&=&(\mathbb T\, \ff)(\bx,t):=\itg_0^t \itg_{\R^3} \mathbf T(\bx-\by,t-\tau)\ff(\by,\tau)d\by d\tau
\\\label{Press}
P(\bx,t)&=&\itg_{\R^3}\nabla E(\bx-\by)\cdot\ff(\by,t)d\by
\end{eqnarray}
(cf. \cite[p.93]{La} or \cite[p.343]{So}, see also \cite{TG}), where
\[
\Gamma(\bx,t)=\frac{\ee^{-|\bx|^2/(4\nu t)}}{(4\pi\nu t)^{3/2}}
\]
is the fundamental solution of the heat equation; 
\[
E(\bx)=-\frac{1}{4\pi |\bx|}\] 
is the fundamental solution of the Laplace equation, hence
\[
\nabla E(\bx)= \frac{1}{4\pi} \frac{\bx}{|\bx|^3};
\]
the matrix \(\mathbf T=\{T_{ij}\}_{i,j=1,2,3}\) is the fundamental solution of the non-stationary Stokes system (the Oseen tensor)
\begin{multline*}
T_{ij}(\bx,t)=\Gamma(\bx,t)\delta_{ij}+\frac{1}{4\pi} \frac{\partial^2}{\partial x_j\partial x_j}\itg_{\R^3} \frac{\Gamma(\bx-\by,t)}{|\by|}d\by\\=
\Gamma(\bx,t)\delta_{ij}+\frac{1}{4\pi} \frac{\partial^2}{\partial x_j\partial x_j}\frac{ 1}{|\bx|}\erf\left(\frac{|\bx|}{2\sqrt{\nu t}}\right)\,
\end{multline*}
where \(\erf(r)=\frac{2}{\sqrt{\pi}}\int_0^r \ee^{-x^2}dx\).

The Stokes equations are the linearized form of the Navier-Stokes equations and they model the simplest incompressible flow problems (cf.\cite{GG}, \cite{GS}): the convection term is neglected, hence the arising model is linear. Thus the difficult is the coupling of velocity \(\bu\) and pressure \(P\).
A great deal of work has been done for the numerical solution of the Stokes and the Navier-Stokes equations, based on the  use of finite element, finite-difference or  finite-volume methods (cf., e.g.,  \cite{GQS},  \cite{GP}, \cite{JV},  \cite{RR}, \cite{RT}). 
Here we propose a fast method for the approximations of \(\bu_1,\bu_2\) and \(P\) in the framework of {\it approximate approximations }(\cite{MSbook}), which provides very efficient high order approximations (\cite{LS2020}).

The outline of the paper is the following. In section \ref{sec1} we consider the approximations of {  \(\bu_1=(u_{11},u_{12},u_{13})\), solution of \eqref{heat}.}  Our method, proposed in \cite{LMS2016} for \(n\)-dimensional parabolic problems, consists in approximating the functions { \(\bg=(g_1,g_2,g_3)\)} via the basis functions introduced by approximate approximations, which are product of Gaussians and special polynomials. The action of the Poisson integral \(\cP\) applied to the basis functions admits a separated representation (also tensor product representation), i.e., it is represented as product of functions depending on one  of the variables. Then a separated representation of \(\bg\) provides a separated representation of the potential and the resulting approximations formulas are very fast because only one-dimensional operations are used.\\
In section \ref{sec2} we consider the approximation of the pressure \(P\) { given in \eqref{Press}} and its gradient \(\nabla P\). For fixed \(t>0\), keeping in mind \eqref{eq2}, the divergence operator applied to equation \eqref{eq1} gives
\[
-\Delta P=F
\]
with \(F(\bx,t)=-\nabla\cdot\,\ff(\bx,t)=-\sum_{j=1}^3 \partial_{x_j} f_j(\bx,t)\). The solution \(P(\bx,t)\) with \(P(\bx,t)\to 0\) as \(|\bx|\to \infty\) is given by
\begin{equation}\label{harmonic}
P(\bx,t)=\cL F(\bx,t)=-\frac{1}{4\pi}\itg_{\R^3} \frac{F(\by,t)}{|\bx-\by|}d\by=
\frac{1}{4\pi} \sum_{j=1}^3 \itg_{\R^3} \frac{\partial_{y_j} f_j(\by,t)}{|\bx-\by|}d\by
\end{equation}
where \(\cL\) denotes the harmonic potential. 
In \cite{LMS2011}  fast semi-analytic cubature formulas for computing \(\cL F\) were constructed which are based on the approximation of the density \(F(\cdot,t)\) by functions with analytically known harmonic potentials. 
The gradient \(\nabla P=\nabla \cL F\) { can be approximated} by the gradient of the cubature of \(\cL F\), which is a semi-analytic formula, too. If \(F(\cdot,t)\) admits a separated representation then we derive tensor product representations of \(P\) and \(\nabla P\), which admits efficient one-dimensional operations.

 In section \ref{sec3} we describe the approximation of { \(\bu_2=(u_{21},u_{22},u_{23})\)}, which satisfy
 the non-homo\-geneous heat equations
\begin{eqnarray}\label{eq2ter}
\frac{\partial \bu_{2}}{\partial t} -\nu \Delta \bu_{2}=\Phi,\quad 
\bu_{2}(\bx,0)=0,\quad 
\bx\in\R^3,\quad t>0,
\end{eqnarray}
where \(\Phi=\ff-\nabla P\) and \(\nabla P=\nabla \cL F\) is computed in section \ref{sec2}. Since \(\nabla \cdot \Phi=0\) then the solution of \eqref{eq2ter} satisfies the condition \(\nabla \cdot \bu_2=0\). The solution of \eqref{eq2ter} admits the representation 
 \[
\bu_{2}(\bx,t)=\cH \Phi(\bx,t)=(\cH \varphi_1, \cH\varphi_2,\cH \varphi_3),\quad \Phi=(\varphi_1,\varphi_2,\varphi_3)
\]
with
\begin{equation}\label{potential2}
\begin{split}
\cH \varphi_j(\bx,t)
&=\itg_0^t \frac{ds}{(4 \pi\,\nu s)^{3/2}}\itg_{\R^3} \ee^{-\frac{|\bx-\by|^2}{4\nu s}} \, \varphi_j(\by,t-s)d\by \\
&=\itg_0^t   (\cP\, \varphi_j(\cdot,s))(\bx,t-s) ds,\quad j=1,2,3
\end{split}
\end{equation}
({cf.} \cite{evans}). The cubature formula of \eqref{potential2}  proposed in \cite{LMS2016} is based on replacing the density \(\varphi_j\) 
 by approximate quasi-interpolants on the rectangular grids . 
The action of \(\cH\) on the basis functions allows for one-dimensional integral representations with separated integrands. This
construction, combined with an accurate quadrature rule as suggested in \cite{Mo} and a separated representation of the density \(\varphi_j\), provides a separated representation of the integral operator \eqref{potential2}. 

We derive fast and
accurate formulas for the approximation of \((\bu_1,\bu_2,P)\) of an arbitrary high order.
In section \ref{sec4} the accuracy of the method and the convergence orders 2,4, 6  and 8 are confirmed by numerical experiments.

\section{Approximation of \(\bu_1\)}\label{sec1}

In this section we consider  the approximation of \(\bu_1=(u_{11},u_{12},u_{13})\). Each component \(u_{1j},j=1,2,3\), is  solution of the  homogeneous heat equation
\begin{eqnarray*}\label{eq2bis}
\frac{\partial u_{1j}}{\partial t} -\nu \Delta u_j=0,\quad
u_{1j}(\bx,0)=g_j(\bx)\quad \bx\in \R^3,\quad t>0.
\end{eqnarray*}
The solution is given by
\[
u_{1j}(\bx,t)=\cP g_j(\bx,t)\,,
\]
where \(\cP g_j\) denotes the Poisson integral
\begin{equation*}\label{poisson}
\cP g_j(\bx,t)
=
\itg_{\R^3}\Gamma(\bx-\by,t)\,g_j(\by)d\by = \frac{1}{(4\nu\pi\,t)^{3/2}}\itg_{\R^3}\ee^{-|\bx-\by|^2/(4\nu t)}\,g_j(\by)d\by \,.
\end{equation*}
We approximate  the {functions  $g_j, j=1,2,3$,} by  the approximate quasi-interpolants  
\begin{equation}\label{quasiintg}
\cN_h g_{j}(\bx) =\frac{1}{\cD^{3/2}}\sum_{\bm\in\Z^3} {g_j}(h \bm)\, 
 {\widetilde\h_{2M}} \left(\frac{\bx-h \bm}{h \sqrt{\cD}}\right)
\end{equation}
with the basis functions
\begin{equation}\label{basis}
\widetilde\h_{2M}(\bx)=\prod_{j=1}^3{\h}_{2M}(x_j);
\quad
{\h}_{2M}(x)=\frac{(-1)^{M-1}}{2^{2M-1}\sqrt{\pi} (M-1)!}\frac{H_{2M-1}(x) \ee^{-x^2}}{x}\,,
\end{equation}
where $H_k$ are the Hermite polynomials
\begin{equation*}\label{hermite}
H_k(x)=(-1)^k \ee^{x^2} \left( \frac{d}{dx}\right)^k \ee^{-x^2}.
\end{equation*}
Here \(\cD\) is a positive fixed parameter and \(h\) is the mesh size.
The function \(\widetilde{\h}_{2M}\) satisfies the moment conditions of order  $2M$  
\begin{equation*}\label{moment}
\itg_{\R^3} \bx^\alpha \widetilde{\h}_{2M}(\bx) d\bx=\delta_{0,\alpha},\qquad 0\leq |\alpha|<N\, ,
\end{equation*}
 and the quasi-interpolant \eqref{quasiintg}
provides an approximation of $g_j$ with the general asymptotic error \(\cO(h^{2M}+\e)\) (cf. \cite{MSbook}). 
The saturation error \(\eps\) does not converge to zero as \(h\to 0 \), however it can be made arbitrary small if the parameter \(\cD\) is sufficiently large.
Then the sum
\begin{equation*}\label{PM}
(\cP_Mg_j )(\bx,t):=(\cP \cN_h g_{j})(\bx,t)
= \frac{1}{\cD^{3/2}}\sum_{{\bm\in\Z^3}} {g_j}(h \bm) \, \cP\widetilde\h_{2M} \,
\left(\frac{\bx-h\bm}{h\sqrt{\cD}},\frac{t}{h^2{\cD}}\right)
\end{equation*}
provides  an approximation of  $\cP\,g_j(\bx,t)$ with the error \(\cO((h\sqrt{\cD})^{2M}+\e)\). Since the Poisson integral is a smoothing operator, one can prove that also the saturation error tends to \(0\) as \(h\to 0\)  and  the approximate solution \(\cP_Mg_j \) converges for any fixed \(t>0\) with order \(\cO(h^{2M})\) to $\cP\,g_j$ (\cite[Theorem 6.1]{MSbook}).
\begin{thm}(\cite[Theorem 3.1]{LMS2016})\label{3.1}
Let \(M\geq 1\). The Poisson integral applied to the generating function \(\widetilde\h_{2M}\) in \eqref{basis} can be written as
\[
\cP\widetilde\h_{2M} (\bx,t)=\frac{\ee^{-\frac{|\bx|^2}{1+4\nu t}}}{\pi^{3/2}}\prod_{j=1}^3 \cQ_M(x_j,4\nu t)\,.
\]
$\cQ_M(x,r)\) is a  polynomial in $x$
of degree $2M-2$ of the form
\begin{equation}\label{QM}
\cQ_M(x,r)=
\sum_{k=0}^{M-1} \frac{1}{(1+ r)^{k+1/2}}
\frac{(-1)^k}{4^k k!}H_{2k}\Big(\frac{x}{\sqrt{1+r}}\Big)\,.
\end{equation}
\end{thm}
From theorem \ref{3.1},  the sum  
\begin{equation}\label{app4}
(\cP_Mg_j )(\bx,t)
= \frac{1}{(\pi\cD)^{3/2}}\sum_{\bm\in\Z^3} {g_j}(h \bm)\ee^{-\frac{|\bx-h\bm|^2}{h^2\cD+4\nu t}}\prod_{i=1}^3 \cQ_M(\frac{x_i-hm_i}{h\sqrt{\cD}},4\nu \frac{t}{h^2{\cD}})
\end{equation}
is a semi-analytic cubature formula for \(\cP g_j\).
For example, for \(M=1\), 
\[
\cQ_1(x,r)=\frac{1}{\sqrt{1+r}}
\]
and then 
\begin{equation*}
(\cP_2 g_{j})(\bx,t)
= \frac{h^3}{\pi^{3/2}}\sum_{\bm\in\Z^3} {g_j}(h \bm)\frac{\ee^{-\frac{|\bx-h\bm|^2}{h^2\cD+4\nu t}}}{(h^2\cD+4\nu t)^{3/2}}.
\end{equation*}
The computation of \eqref{app4} is very efficient if  the functions {\(g_j(\bx),j=1,2,3,\)} allow a separated representation; that is, within a prescribed accuracy \(\e\), they can be represented as sum of products of univariate functions
\[
g_j(\bx)=\sum_{s=1}^S\prod_{i=1}^3 g_{j,i}^{(s)}(x_i)+\cO(\e)
\]
with suitable functions \(g_{j,i}^{(s)}\). Then, at the point of a uniform grid \(\{h\bk\}\),
\begin{equation}\label{hg}
(\cP g_{j})(h\bk,t)\approx \sum_{s=1}^S\prod_{i=1}^3 S_{j,s}^{(i)}(k_i,\frac{4\nu t}{h^2\cD})
\end{equation}
where \(S_{j,s}^{(i)}\) denotes the one-dimensional convolution
\[
S_{j,s}^{(i)}(k,t)={(\pi \cD)^{-1/2}}\sum_{m\in\Z} g_{j,i}^{(s)}(hm)\ee^{-\frac{(k-m)^2}{\cD(1+t)}} \cQ_M(\frac{k-m}{\sqrt{\cD}},t).
\]

\section{Approximation of \(P(\bx,t)\)}\label{sec2}

For fixed \(t>0\), we consider the approximation of \(P(\cdot,t)=\cL F(\cdot,t)\), given in \eqref{harmonic}.
We approximate \(F\) by the quasi-interpolant
\begin{equation}\label{quasiint}
\cM_{h} F(\bx,t)={\cD^{-3/2}}\sum_{\bm\in\Z^3}F(h\bm,t) \widetilde\eta_{2M}\left(\frac{\bx-h\bm}{h\sqrt{\cD}}\right)
\end{equation}
with the basis functions \eqref{basis}.
The quasi-interpolant \eqref{quasiint}
provides an approximation of $F$ with the general asymptotic error \(\cO(h^{2M}+\e)\) (cf. \cite[Theorem 4.2]{MSbook}). 
Then the sum
\[
\cL_M F(\bx,t):=\cL \cM_h F(\bx,t)= \frac{h^2}{\cD^{1/2}} \sum_{\bm\in\Z^3} F(h\bm,t) (\cL\widetilde\eta_{2M})\left(\frac{\bx-h\bm}{h\sqrt{\cD}}\right)\]
\[={-}
\frac{h^2}{\cD^{1/2}} \sum_{i=1}^3\sum_{\bm\in\Z^3} \partial_{x_i} f_i(h\bm,t) (\cL\widetilde\eta_{2M})\left(\frac{\bx-h\bm}{h\sqrt{\cD}}\right)
\]
is a semi-analytic cubature formula for \(P=\cL F\). 
Moreover, by the smoothing properties of the harmonic potential, the corresponding small saturation error {converges} with the rate \(h^2\) as \(h\to 0\). Hence, 
 for sufficiently smooth  functions, \(\cL_M F\) approximates \(\cL F\) with the error 
\(
\cO(h^{2M}+h^2\ee^{-\cD\pi^2})\,
\)
(\cite[Theorem 4.10]{MSbook}).Therefore, in numerical computations with \(\cD\geq 3\), \(\cL_M F(\bx,t)\) behaves like a high order cubature formula.\\
The function \(v= \cL(\prod_{j=1}^3{\h}_{2M})\) is solution of the equation
\begin{equation}\label{advecM}
-\Delta \, v = \prod_{j=1}^3 \eta_{2M}(x_j)\,,\quad \bx\in\R^3,\quad |v(\bx)|\to 0\quad as  \quad |\bx|\to \infty\,.
\end{equation}
\begin{thm}(\cite[Theorem 1]{LMS2014})
The solution of \eqref{advecM}  can be expressed by the one-dimensional integral
\begin{align} \label{zwint}
v(\bx)
=
\frac{1}{4\pi^{3/2}} \itg_0^\infty {\ee^{-\frac{|\bx|^2}{1+ r}}}\prod_{j=1}^3\cQ_M(x_j,r)dr\,,
\end{align}
where $\cQ_M(x,r)\) is defined in \eqref{QM}.
\end{thm}
For example, for \(M=1\), 
{ \[
\cQ_1(x,r)=\frac{1}{\sqrt{1+r}},\quad
\cL(\ee^{-|\cdot|^2})(\bx)=\frac{1}{4\pi^{3/2}} \itg_0^\infty \frac{\ee^{-|\bx|^2/(1+ r)}}{(1+r)^{3/2}} dr\,
\]
and
\[
P(\bx,t)\approx \cL_M F(\bx,t)=\frac{h^2\cD}{4(\pi \cD)^{3/2}} \sum_{\bm\in\Z^3} F(h\bm,t)\itg_0^\infty \frac{\ee^{-\frac{|\bx-h\bm|^2}{h^2\cD(1+r)}}}{(1+r)^{3/2}}dr
\]
}
provides a second order approximation formula.
\\
The gradient \(\nabla P=\nabla \cL F\) of the harmonic potential can be approximated by the gradient of the cubature of \(\cL F\), which is a semi-analytic formula, too. We have
\[\nabla P(\bx,t)\approx\nabla(\cL_M F)(\bx,t)=
\frac{h}{\cD}  \sum_{\bm\in\Z^3}F(h\bm,t)  \left(\nabla \cL(\prod_{j=1}^3{\h}_{2M}(x_j))\right)\left(\frac{\bx-h\bm}{h\sqrt{\cD}}\right)\,.
\]
Moreover, \(\nabla(\cL_M F)\) approximates \(\nabla(\cL F)\) with the error 
\(
\cO(h^{2M}+h\ee^{-\pi^2\cD})
\) (\cite[Theorem 4.11]{MSbook}.)\\
 The function \(v=\partial_{x_i} \cL(\prod_{j=1}^3{\h}_{2M})\), {\(i=1,2,3\)} is solution of the problem
\begin{equation}\label{advec2}
 -\Delta \, v = \partial_{x_i}\prod_{j=1}^3 \eta_{2M}(x_j)\,,\quad \bx\in\R^3,\quad |v(\bx)|\to 0\quad {\rm  as}  \quad |\bx|\to \infty.
\end{equation}
\begin{thm} 
The solution of \eqref{advec2} can be expressed by the one-dimensional integral
\begin{multline}\label{gradL}
\partial_{x_i} \cL(\prod_{j=1}^3{\h}_{2M})(\bx)=-\frac{x_i}{2\pi^{3/2}} \itg_0^\infty \frac{{\ee^{-\frac{|\bx|^2}{1+ r}}}}{1+r} \prod_{j=1}^3\cQ_M(x_j,r)dr\\+
\frac{1}{4\pi^{3/2}} \itg_0^\infty {\ee^{-\frac{|\bx|^2}{1+ r}}}  \cR_M(x_i,r) \prod_{\substack{j=1\\ j\neq i}}^{3}\cQ_M(x_j,r) dr\,,
\end{multline}
where $\cQ_M(x,r)\) is defined in \eqref{QM}, {\(\cR_1(x,r)=0\) and
\[
 \cR_M(x,r)=\partial_{x}\cQ_M(x,r)=
\sum_{k=1}^{M-1} \frac{1}{(1+ r)^{k+1}}
\frac{(-1)^k}{4^{k-1} (k-1)!}H_{2k-1}\Big(\frac{x}{\sqrt{1+r}}\Big), M>1.
\]
}
\end{thm}
\begin{proof}
Formula \eqref{gradL} can be obtained by direct differentation of \eqref{zwint}, { keeping in mind the identity \(H_{2k}^{\prime}(x)=4 k H_{2k-1}(x)\) (see \cite[(14), p. 193]{emo}).}
\end{proof}
For example, for \(M=1\)
\[
 \cR_1(x,{ r})=\partial_{x}\cQ_1(x,{ r})=0\,,\quad
\nabla \cL(\ee^{-|\cdot|^2})(\bx)=-\frac{\bx}{2\pi^{3/2}} \itg_0^\infty \frac{\ee^{-\frac{|\bx|^2}{1+ r}}}{(1+r)^{5/2}} dr
\]
and 
\[
\nabla P(\bx,t)\approx \nabla { \cL_M F} (\bx,t)=-\frac{h\sqrt{\cD}}{{ 2} (\pi \cD)^{3/2}} \sum_{\bm\in\Z^3} F(h\bm,t)\frac{\bx-h\bm}{h\sqrt{\cD}}\itg_0^\infty \frac{\ee^{-\frac{|\bx-h\bm|^2}{h^2\cD(1+r)}}}{(1+r)^{5/2}}dr.
\]
The quadrature of the integrals  \eqref{zwint} and \eqref{gradL}  with certain quadrature weights \(\omega_p\) and nodes \(r_p\) gives the separated representations
\begin{equation}\label{P}
P(\bx,t)\approx \frac{h^2\cD}{4(\pi\cD)^{3/2}} \sum_{\bm\in\Z^3} F(h\bm,t)  \sum_p \omega_p{\ee^{-\frac{|\bx-h\bm|^2}{h^2\cD(1+r_p)}}}\prod_{j=1}^3\cQ_M(\frac{x_j-hm_j}{h\sqrt{\cD}},r_p)\,,
\end{equation}
\begin{multline}\label{gradP}
\partial_{x_i} P(\bx,t)\approx { - \frac{h\sqrt{\cD}}{2 (\pi \cD)^{3/2}} \sum_{\bm\in\Z^3}F(h\bm,t) \sum_p \omega_p
\Big(
\frac{x_i-hm_i} {h\sqrt{\cD}}{\ee^{-\frac{|\bx-h\bm|^2}{h^2\cD(1+r_p)}}}\prod_{j=1}^3\cQ_M(\frac{x_j-hm_j} {h\sqrt{\cD}},r_p)}
\\+
{\ee^{-\frac{|\bx-h\bm|^2}{h^2\cD(1+r_p)}}}\cR_M(\frac{x_i-hm_i} {h\sqrt{\cD}},r_p) \prod_{\substack{j=1\\ j\neq i}}^{3}\cQ_M(\frac{x_j-hm_j} {h\sqrt{\cD}},r_p)
\Big)\,.
\end{multline}
The approximation formulas \eqref{P} and \eqref{gradP} are very efficient if \(F(\bx,t)={ -}\nabla\cdot\ff(\bx,t)\)  has separated representation, i.e. for given accuracy \(\e\) it can be represented as a sum of product of vectors in dimension \(1\)
\[
F(\bx,t)=\sum_{s=1}^S a_s\prod_{j=1}^3 F_j^{(s)}(x_j,t)+\cO(\e).
\]
Then an approximate value of \(P(h\bk,t)\) and \(\nabla P(h\bk,t)\)  at the point of a uniform grid can be computed by the sum of products of one-dimensional convolutions
\begin{equation}\label{Phk}
P(h\bk,t)
\approx   \frac{h^2\cD}{4(\pi\cD)^{3/2}} \sum_{s=1}^S a_s  \sum_p \omega_p \prod_{j=1}^3 \Sigma_j^{(s)}(k_j,r_p,t)\,,
\end{equation}
\begin{multline} \label{gradPhk}
\partial_{x_i} P(h\bk,t)\approx \frac{h\sqrt{\cD}}{4(\pi\cD)^{3/2}}  \sum_{s=1}^S a_s  \sum_p \omega_p  \Big(T_i^{(s)}(k_i,r_p,t)\prod_{\substack{j=1\\ j\neq i}}^3 \Sigma_j^{(s)}(k_j,r_p,t)
\\+R_i^{(s)}(k_i,r_p,t)\prod_{\substack{j=1\\ j\neq i}}^3 \Sigma_j^{(s)}(k_j,r_p,t)
\Big),\quad i=1,2,3
\end{multline}
with the one-dimensional convolutions
\begin{eqnarray*}
 \Sigma_j^{(s)}(k,r,t)=\sum_{m\in\Z} F_j^{(s)}(hm,t)\cQ_M(\frac{k-m}{\sqrt{\cD}},r) \ee^{-\frac{(k-m)^2}{\cD(1+r)}},
\\
 T_j^{(s)}(k,r,t)=-2\sum_{m\in\Z} F_j^{(s)}(hm,t) \frac{k-m}{\sqrt{\cD}}\cQ_M(\frac{k-m}{\sqrt{\cD}},r) \ee^{-\frac{(k-m)^2}{\cD(1+r)}},
\\
 R_j^{(s)}(k,r,t)=\sum_{m\in\Z} F_j^{(s)}(hm,t) \cR_M(\frac{k-m}{\sqrt{\cD}},r) \ee^{-\frac{(k-m)^2}{\cD(1+r)}}.
\end{eqnarray*}

\section{Approximation of \(\bu_2\)}\label{sec3}

In this section we consider the approximation of  the solution \(\bu_2=(u_{21},u_{22},u_{23})\) of \eqref{eq2ter}. Each component \(u_{2j}\), \(j=1,2,3\),  satisfies the non-homogeneous heat equation
\begin{eqnarray}\label{eq2qua}
\frac{\partial u_{2j}}{\partial t} -\nu \Delta u_{2j}=\varphi_j,\quad 
u_{2j}(\bx,0)=0,\quad 
\bx\in\R^3,t>0\,,
\end{eqnarray}
where \(\varphi_j=f_j -\partial_{x_j} P\). Hence,
\[
u_{2j}(\bx,t)=\cH \varphi_j(\bx,t)
\]
with
\begin{equation*}\label{potential3}
\cH \varphi_j(\bx,t)
=\itg_0^t   (\cP\, \varphi_j(\cdot,s))(\bx,t-s) ds.
\end{equation*}
We replace the density \(\varphi_j\) 
 by the quasi-interpolant on the rectangular grid
$\{(h \bm,\tau i)\}$ 
\begin{equation*}\label{fht}
\cM_{h,\tau} \varphi_j(\bx,t) =\frac{1}{{\cD_0^{1/2} \cD^{3/2}}}\sum_{\substack
{i\in\Z\\\bm\in\Z^3}}\varphi_j(h \bm,\tau i)\, {\h}_{2M}
\left(\frac{t-\tau  i}{\tau \sqrt{\cD_0}}\right) {\widetilde\h}_{2M} \left(\frac{\bx-h \bm}{h \sqrt{\cD}}\right) \,.
\end{equation*}
Here $\tau$ and $h$ are the steps; $\cD_0$ and $\cD$ are positive fixed parameters; 
$\widetilde{\h}_{2M}$ and $\h_{2M}$ are the generating functions given in \eqref{basis}. The sum \(\cM_{h,\tau} \varphi_j\) approximates \(\varphi_j\) with the error
\[
|\varphi_j(\bx,t)-\cM_{h,\tau} \varphi_j(\bx,t)|=\cO((h\sqrt{\cD})^{2M}+(\tau\sqrt{\cD_0})^{2M})+\e
\]
for all \((\bx,t)\in \R^3\times [0,T], T>0\) and the sum
\[
\cH \cM_{h,\tau} \varphi_j(\bx,t)=
\frac{1}{\cD_0^{1/2} \cD^{3/2}}\sum_{\substack
{i\in\Z\\\bm\in\Z^3}} {\varphi_j}(h
\bm,\tau i) 
\itg_0^t   {\h_{2M}}\Big(\frac{s-\tau i}{\tau \sqrt{\cD_0}}\Big)
\, \cP \widetilde\eta_{2M}  \,
\Big(\frac{\bx-h\bm}{h\sqrt{\cD}},\frac{t-s}{h^2{\cD}}\Big) ds
\]
approximates \(\cH \varphi_j\) in \(\R^3\times [0,T]\) with the error estimate
\[
|\cH \cM_{h,\tau} \varphi_j(\bx,t)-\cH  \varphi_j(\bx,t)|=\cO((h\sqrt{\cD})^{2M}+(\tau\sqrt{\cD_0})^{2M})+\epsilon
\]
{ (see \cite[Theorem. 2.1]{LMS2014}).}
Here
\[
{\varphi_j}(h\bm,\tau i) =f_j(h\bm,\tau i) -\partial_{x_j} P(h\bm,\tau i)
\]
and \(\partial_{x_j} P(h\bm,\tau i)\) is given in \eqref{gradPhk}.

At the point of a rectangular grid \(\{(h\bk,\tau \ell)\}\) we get 
\begin{equation}\label{approxheat}
\cH \cM_{h,\tau} \varphi_j(h\bk,\tau \ell)=
\frac{1}{\cD_0^{1/2} \cD^{3/2}}\sum_{\substack
{i\in\Z\\\bm\in\Z^3}}  {\varphi_j}(h
\bm,\tau i) 
K_{M}(h\bk,h\bm,\tau\ell,\tau i)
\end{equation}
where 
\[
K_{M}(h\bk,h\bm,\tau\ell,\tau i)=\itg_0^{\tau \ell}  {\h_{2M}}\Big(\frac{\sigma-\tau i}{\tau \sqrt{\cD_0}}\Big)
\, \cP \widetilde\eta_{2M}  \,
\Big(\frac{\bk-\bm}{\sqrt{\cD}},\frac{\tau \ell-\sigma}{h^2{\cD}}\Big) d\sigma.
\]
In view of theorem \ref{3.1}
\begin{multline*}
K_{M}(h\bk,h\bm,\tau\ell,\tau i)=\\
{ \pi^{-3/2}}\itg_0^{\tau \ell}  {\h_{2M}}\Big(\frac{\sigma-\tau i}{\tau \sqrt{\cD_0}}\Big)
\prod_{j=1}^3
\ee^{-(k_j-m_j)^2/(\cD(1+4\nu \frac{\tau \ell-\sigma}{h^2\cD}))}
\cQ_M(\frac{k_j-m_j}{\sqrt{\cD}},4\nu \frac{\tau \ell-\sigma}{h^2{\cD}})d\sigma
\end{multline*}
and the integrals cannot be taken analytically. The computation of the sum \eqref{approxheat} involves additionally an integration, which must be approximated by an efficient quadrature rule with certain quadrature weights \(\omega_p\) and nodes \(r_p\)
\begin{multline*}
K_{M}(h\bk,h\bm,\tau\ell,\tau i)\approx
\\ { \pi^{-3/2}}\sum_p \omega_p \  {\h_{2M}}\Big(\frac{\tau_p-\tau i}{\tau \sqrt{\cD_0}}\Big)
\prod_{j=1}^3
\ee^{-(k_j-m_j)^2/(\cD(1+4\nu \frac{\tau \ell-\tau_p}{h^2\cD}))}
\cQ_M(\frac{k_j-m_j}{\sqrt{\cD}},4\nu \frac{\tau \ell-\tau_p}{h^2{\cD}})\,.
\end{multline*}

If \(\Phi=(\varphi_1,\varphi_2,\varphi_3)\) admits a separate representation, then the sum \eqref{approxheat} gives an efficiently computable high order approximation of the initial value problem \eqref{eq2qua} based on the computation of one-dimensional sums.

\section{Implementation and numerical experiments}\label{sec4}

\subsection{Homogeneous heat equation}

We assume \(\ff=\bf0\) and
\(
\bg(\bx)=\nabla \times (0,0,\ee^{-|\bx|^2})=(-2x_2\ee^{-|\bx|^2},2x_1 \ee^{-|\bx|^2},0).
\)
Then 
\(
\nabla \cdot \bg(\bx)=0
\)
and the solution of \eqref{heat} -\eqref{heat2} is provided by \(\bu_2\equiv 0\), \(P\equiv0\) and \(\bu=\bu_1=(u_{11},u_{12},u_{13})\) with
\[
u_{11}(\bx,t)=(\cP g_1)(\bx,t)=\frac{-2}{(4\pi\nu t)^{3/2}}\itg_{\R^3}y_2 \ee^{-\frac{|\by-\bx|^2}{4\nu t}}\ee^{-|\by|^2}d\by=-2 {x_2} \frac{\ee^{-\frac{|\bx|^2}{1+4\nu t}}}{(4 \nu  t+1)^{5/2}}\,,
\]
\[
u_{12}(\bx,t)=(\cP g_2)(\bx,t)=\frac{2}{(4\pi\nu t)^{3/2}}\itg_{\R^3}y_1 \ee^{-\frac{|\by-\bx|^2}{4\nu t}}\ee^{-|\by|^2}d\by=2 {x_1} \frac{\ee^{-\frac{|\bx|^2}{1+4\nu t}}}{(4 \nu  t+1)^{5/2}}\,,
\]
\[
u_{13}(\bx,t)=(\cP g_3)(\bx,t)=0\,.\]
We provide results of some experiments which show accuracy and numerical convergence order of the { approximation formulas \eqref{hg}.} In Tables \ref{T1}-\ref{T1bis} we compare the exact solution \(\cP g_1\) with
the coefficient of kinematic viscosity  \(\nu\) equals to \(1\) 
and the approximate solution in \eqref{hg} at one fixed point, for \(M=1,2,3,4\) and  different values of \(h\). 
{We choose \(\cD=4\) to have the saturation error comparable with the double precision rounding errors. 
Numerical experiments show that the predicted convergence order is obtained and for \(M=4\) and small \(h\) the saturation error is reached.}

 \begin{table}[h]
{\small 
\begin{center}
\begin{tabular}{c|cc|cc|cc|cc} 
& \multicolumn{2}{c|}{$M=1$}&\multicolumn{2}{c}{$M=2$}&\multicolumn{2}{c}{$M=3$}&\multicolumn{2}{c}{$M=4$}\\
$h^{-1}$&  error &      rate&  error &      rate&  error &      rate&  error &      rate          \\\midrule
10&0.235E-03&&0.133E-05&& 0.742E-08&&0.367E-10&\\
20&0.591E-04 &   1.99&0.841E-07 &   3.98&0.117E-09 &   5.98&0.145E-12 &   7.98\\
40&0.148E-04 &   1.99&0.841E-09 &   3.99&0.184E-11 &   5.99&0.586E-15 &   7.94\\
80&0.370E-05 &   1.99&0.329E-09 &   3.99&0.288E-13 &   5.99&0.694E-17&\\
160&0.925E-06 &   1.99&0.206E-10 &   3.99&0.545E-15 &   5.72&0.902E-16&\\
   \bottomrule
   \end{tabular}
   \caption{Absolute error and rate of convergence for $\cP g_1$ in \(\bx=(1.2,1.2,1.2), t=1\)  using $\cP_M g_1$.}\label{T1}
    \end{center}
    }
    
    {\small 
\begin{center}
\begin{tabular}{c|cc|cc|cc|cc} 
& \multicolumn{2}{c|}{$M=1$}&\multicolumn{2}{c}{$M=2$}&\multicolumn{2}{c}{$M=3$}&\multicolumn{2}{c}{$M=4$}\\
$h^{-1}$&  error &      rate&  error &      rate&  error &      rate&  error &      rate          \\\midrule
10& 0.540E-03& &0.321E-05 & &0.175E-07 & &0.809E-10 & \\
20& 0.136E-03 &   1.98 & 0.202E-06 &   3.98 & 0.276E-09 &   5.98 &0.319E-12 &   7.98 \\
40& 0.341E-04 &   1.99& 0.127E-07 &   3.99 & 0.432E-11 &   5.99 & 0.128E-14 &   7.96 \\
80& 0.852E-05 &   1.99 & 0.791E-09 &   3.99 & 0.676E-13 &   5.99 &0.763E-16  & \\
160& 0.213E-05 &   1.99 & 0.494E-10 &   3.99 & 0.117E-14 &   5.85 & 0.125E-15& \\
   \bottomrule
   \end{tabular}
   \caption{Absolute error and rate of convergence for $\cP g_1$ in \(\bx=(0,1.6,0), t=1\)  using $\cP_M g_1$.}\label{T1bis}
    \end{center}
    }

     \end{table}

\bigskip\bigskip

\subsection{Numerical results for  the approximation of \(P\) and \(\nabla P\) }

We assume \(\ff(\bx,t)={ 2} t\, \bx \ee^{-|\bx|^2}=(2t\,x_1\ee^{-|\bx|^2},2t\,x_2\ee^{-|\bx|^2},2t\,x_3\ee^{-|\bx|^2})\) and \(\bg(\bx)={(0,0,0)}\).
Hence,  
\[
\nabla\cdot \ff=t(6-4|\bx|^2)\ee^{-|\bx|^2}.
\]
The unique solution of \eqref{eq1}, \eqref{eq2}, \eqref{eq3} is
\[
\bu(\bx,t)={\bf 0},\quad P(\bx,t)=-t\,\ee^{-|\bx|^2}.
\]
We report on the absolute errors and the approximation rates for the harmonic potential \(P=\cL F\) ( Tables  \ref{T3} and \ref{T3bis}) and {\(\partial_{x_2} P=2tx_2 \ee^{-|\bx|^2}\)} (Tables \ref{T5} and \ref{T6}) in a fixed point. The approximate values are computed by the  cubature formulas { given in} \eqref{Phk} having the approximation order \(\cO(h^{2M}+h^2\ee^{-\cD\pi^2})\), and   { in} \eqref{gradPhk} having the approximation order \(\cO(h^{2M}+h\ee^{-\cD\pi^2})\), respectively, for \(M=1,2,3,4\).  { We used uniform grids size  \(h=0.1\cdot 2^{1-k}\), \(k=1,...,5\) and we choose the parameter \(\cD=5\).} Following \cite{Mo} the one-dimensional integrals  { in} \eqref{zwint} and \eqref{gradL} are transformed to integrals over \(\R\) with integrands decaying doubly exponentially by making the substitutions
\[
t=\ee^\xi,\quad \xi=a (\tau+\ee^\tau),\quad \tau=b (u-\ee^{-u})
\]
with certain positive constants \(a\) and \(b\). The computation is based on the classical trapezoidal rule with step size \(\kappa\),
exponentially converging for rapidly decaying smooth functions on the real line.
In our computations we assumed \(a=5\), \(b=6\), \(\kappa=0.0009\) and \(8\cdot 10^2\) points in the quadrature formula in order to reach the saturation error with the approximation formula of order  \(N=8\) .

 The numerical results show that higher order cubature formulas gives essentially better approximation than the second order formulas and the predicted convergence order is reached.

 \begin{table}[p]
{\small 
\begin{center}
\begin{tabular}{c|cc|cc|cc|cc} 
& \multicolumn{2}{c|}{$M=1$}&\multicolumn{2}{c}{$M=2$}&\multicolumn{2}{c}{$M=3$}&\multicolumn{2}{c}{$M=4$}\\
$h^{-1}$&  error &      rate&  error &      rate&  error &      rate&  error &      rate          \\\midrule
10&0.188D-02&& 0.718D-04&&0.731D-05&&0.960D-08&\\
20& 0.470D-03  & 2.00&0.462D-05  &   3.95&0.143D-06  &   5.68&0.783D-10  &    6.93\\
40&0.117D-03  &     2.00&0.291D-06  &    3.99&0.236D-08   &   5.92&0.352D-12  &    7.79\\
80&0.293D-04    &   2.00&0.182D-07  & 3.99&0.373D-10 &   5.98&0.116D-14  &   8.24\\
160&0.733D-05 &   2.00&0.114D-08  &   3.99&0.585D-12  &   5.99& 0.218D-13&\\
   \bottomrule
   \end{tabular}
   \caption{Absolute error and rate of convergence for $P(\bx,t)$ in \(\bx=(1.2,1.2,1.2), t=1\)  using $\eqref{Phk}$.}\label{T3}
    \end{center}
    }
        
  {\small 
\begin{center}
\begin{tabular}{c|cc|cc|cc|cc} 
& \multicolumn{2}{c|}{$M=1$}&\multicolumn{2}{c}{$M=2$}&\multicolumn{2}{c}{$M=3$}&\multicolumn{2}{c}{$M=4$}\\
$h^{-1}$&  error &      rate&  error &      rate&  error &      rate&  error &      rate          \\\midrule
10&0.386D-02&& 0.811D-04&& 0.127D-04 && 0.656D-06&\\
20& 0.101D-02  &   1.93&0.633D-05  &   3.68&0.223D-06  &  5.83& 0.284D-08    &   7.85\\
40& 0.255D-03  &   1.98&0.417D-06 &   3.92& 0.358D-08  &  5.96& 0.114D-10  &   7.96\\
80& 0.640D-04  &   1.99&0.264D-07  &   3.98&0.564D-10   &   5.99& 0.445D-13  &   7.99\\
160&0.160D-04 &  1.99& 0.165D-08  &   3.99& 0.882D-12  & 5.99& 0.194D-15 &   \\   \bottomrule
   \end{tabular}
   \caption{Absolute error and rate of convergence for $P(\bx,t)$ in \(\bx=(0,1.6,0), t=1\)  using $\eqref{Phk}$.}\label{T3bis}
    \end{center}
    }
  \end{table}
     
      \begin{table}[h!]
{\small 
\begin{center}
\begin{tabular}{c|cc|cc|cc|cc} 
& \multicolumn{2}{c|}{$M=1$}&\multicolumn{2}{c}{$M=2$}&\multicolumn{2}{c}{$M=3$}&\multicolumn{2}{c}{$M=4$}\\
$h^{-1}$&  error &      rate&  error &      rate&  error &      rate&  error &      rate          \\\midrule
10& 0.279D-02&& 0.163D-03&& 0.584D-05&&0.140D-06 &\\
20&0.719D-03  &  1.96& 0.107D-04 &    3.92 & 0.961D-07  &    5.92 & 0.543D-09  &    8.01\\
40& 0.181D-03  &  1.99& 0.678D-06 &    3.98&0.152D-08   &    5.98 & 0.211D-11 &    8.01 \\
80&0.454D-04 &    2.00&  0.425D-07  &    4.00& 0.238D-10 &    6.00 & 0.826D-14  &    8.00\\
160& 0.113D-04&    2.00&  0.266D-08  &    4.00& 0.373D-12  & 6.00 &  0.139D-16 &\\
   \bottomrule
   \end{tabular}
   \caption{Absolute error and rate of convergence for $\partial_{x_2}P(\bx,t)$ in \(\bx=(1.2,1.2,1.2),\> t=1\)  using $\eqref{gradPhk}$.}\label{T5}
    \end{center}
    }
{\small 
\begin{center}
\begin{tabular}{c|cc|cc|cc|cc} 
& \multicolumn{2}{c|}{$M=1$}&\multicolumn{2}{c}{$M=2$}&\multicolumn{2}{c}{$M=3$}&\multicolumn{2}{c}{$M=4$}\\
$h^{-1}$&  error &      rate&  error &      rate&  error &      rate&  error &      rate          \\\midrule
10& 0.175D-04 && 0.319D-03 && 0.390D-05&& 0.141D-05 &\\
20&0.750D-02   &    1.98& 0.195D-04 &    4.03 & 0.382D-06   &    5.84 & 0.686D-08   &    7.68 \\
40& 0.188D-02  &    2.00& 0.121D-05 &    4.01 & 0.613D-08  &    5.96 & 0.283D-10 &    7.92\\
80& 0.471D-03  &    2.00&  0.753D-07  &    4.00 & 0.964D-10   &    5.99 &0.112D-12   &    7.97 \\
160& 0.118D-03 &    2.00 &0.470D-08  &    4.00 & 0.151D-11  &    6.00  & 0.416D-15  &  \\  
   \bottomrule
   \end{tabular}
   \caption{Absolute error and rate of convergence for $\partial_{x_2}P(\bx,t)$ in \(\bx=(0,1.6,0), t=1\)  using $\eqref{gradPhk}$.}\label{T6}
    \end{center}
    }
     \end{table}

\subsection{ Non-Homogeneous heat equation}\label{5.3}

In Tables \ref{T7} and \ref{T7bis} we report on the absolute errors and the approximation rates  
for the solution \eqref{eq2ter}  { with \(\nu=1\)}. We assumed  \(\Phi=(\varphi_1,\varphi_2,\varphi_3)=(\ee^{-|\bx|^2} (1+6t-4|\bx|^2 t),0,0)\) which gives the exact  {solution } \(\bu_2=(t\ee^{-|\bx|^2},0,0)\). The approximations have been computed by 
\(\cH\!\! \cM_{h,\tau}\)  in \eqref{approxheat} for \(M=1,2,3,4\), with the parameters \(\cD=\cD_0=4\).

Making the substitution
\begin{equation*}\label{submo}
\sigma=\frac{\tau \ell}{2} \Big(1+\tanh \big( \frac{\pi}{2} \sinh \x \big) \Big)
= \frac{\tau \ell}{1+\ee^{-\pi \sinh \x}} 
\end{equation*}
introduced in \cite{Mo},
$K_M$ transforms to an  integral over $\R$ with 
doubly exponentially decaying integrand. Then we apply the classical trapezoidal rule  with the parameters \(\kappa=0.002\) and \(2\cdot 10^3\) terms  in the quadrature formula in order to reach the saturation error with the approximation formula of order  \(N=8\) .

 \begin{table}[h]
{\small 
\begin{center}
\begin{tabular}{cc|cc|cc|cc|cc} 
&& \multicolumn{2}{c|}{$M=1$}&\multicolumn{2}{c}{$M=2$}&\multicolumn{2}{c}{$M=3$}&\multicolumn{2}{c}{$M=4$}\\
$h^{-1}$&$\tau^{-1}$&  error &      rate&  error &      rate&  error &      rate&  error &      rate          \\\midrule
10&40& 0.151E-02&&0.463E-04&& 0.771E-06 &&0.497E-08&\\
20&80&0.376E-03 &   2.00&0.296E-05 &   3.97&0.118E-07 &   6.02&0.333E-10 &   7.22\\
40&160&0.938E-04 &   2.00&0.186E-06 &   3.99&0.183E-09 &   6.00&0.146E-12 &   7.84\\
80&320&0.234E-04 &   2.00&0.117E-07 &   3.99&0.286E-11 &   6.00&0.671E-15 &   7.76\\
160&640&0.586E-05 &   2.00&0.729E-09 &   3.99& 0.445E-13 &   6.00& 0.121E-15&\\
   \bottomrule
   \end{tabular}
   \caption{Absolute error and rate of convergence for the approximation of $\cH \varphi_1(\bx,t)$ defined in \eqref{potential2}  in \(\bx=(1.2,1.2,1.2), t=1\)  using \eqref{approxheat}.}\label{T7}
    \end{center}
    }
    
    {\small 
\begin{center}
\begin{tabular}{cc|cc|cc|cc|cc} 
&& \multicolumn{2}{c|}{$M=1$}&\multicolumn{2}{c}{$M=2$}&\multicolumn{2}{c}{$M=3$}&\multicolumn{2}{c}{$M=4$}\\
$h^{-1}$&$\tau^{-1}$&  error &      rate&  error &      rate&  error &      rate&  error &      rate          \\\midrule
10&40&0.313E-02&&0.552E-04&&0.671E-05&& 0.276E-06&\\
20&80&0.810E-03 &   1.95&0.411E-05 &   3.75& 0.115E-06 &   5.87&0.117E-08 &   7.88\\
40&160&0.204E-03 &   1.99&0.268E-06 &   3.94&0.184E-08 &   5.97& 0.468E-11 &   7.97\\
80&320&0.512E-04 &   1.99&0.169E-07 &   3.98&0.289E-10 &   5.99&0.183E-13 &   7.99\\
160&640&0.128E-04 &   1.99&0.106E-08 &   3.99&0.452E-12 &   5.99&0.389E-15&\\
   \bottomrule
   \end{tabular}
   \caption{Absolute error and rate of convergence for the approximation of $\cH \varphi_1(\bx,t)$ defined in \eqref{potential2}  in  \(\bx=(0,1.6,0), t=1\)  using \eqref{approxheat}.}\label{T7bis}
    \end{center}
    }

     \end{table}   
     
\newpage

\section*{Acknowledgement}
The second author was supported by the RUDN University Program 5-100.


\begin{thebibliography}{100}

{ \bibitem{emo} A. Erd\'ely, W. Magnus, F. Oberhettinger, F.G. Tricomi, Higher transcendental functions, New York, Mc Graw Hill 1953. }

\bibitem{evans} L.C. Evans, Partial Differential Equations, 2nd edition, Graduate Studies in Mathematics, vol.19, American Mathematical Society, Providence, RI, 2010.

\bibitem{GG} G. Galdi, An Introduction to the Mathematical Theory of the Navier-Stokes Equations, Springer 2011.

\bibitem{GQS} P. Gervasio, A. Quarteroni, F. Saleri, Spectral Approximation of Navier-Stokes Equations, In: Galdi G.P., Heywood J.G., Rannacher R. (eds) Fundamental Directions in Mathematical Fluid Mechanics. Advances in Mathematical Fluid Mechanics. Birkh\"auser, Basel 2000.

\bibitem{GP} R. Glowinski, O. Pironneau, On Numerical Methods for Stokes Problem, Publications mathematiques et informatique de Rennes, no. S4 (1978), 29 p.


\bibitem{JV} V. John, Finite Element Methods for Incompressible flow problems, Springer, 2016.

\bibitem{La} O.A. Ladyzhenskaya, The Mathematical Theory of Viscous Incompressible Flow, Gordon and Breach Science Publishers 1969.

  \bibitem{LMS2011} {F.~Lanzara, V.~Maz'ya , G.~Schmidt}, On the fast computation of high dimensional volume potentials, Math. Comput., 80,  887-904  (2011).

\bibitem{LMS2014} {F.~Lanzara, V.~Maz'ya, G.~Schmidt}, Fast cubature of volume potentials over rectangular domains by approximate approximations, Appl. Comput. Harmon. Anal. 36, 167-182  (2014) .

\bibitem{LMS2016} {F.~Lanzara, V.~Maz'ya, G.~Schmidt}, Approximation of solutions to multidimensional parabolic equations by approximate approximations,  Appl. Comput. Harmon. Anal. , 41, 749--767 (2016) .

\bibitem{LS2020} F. Lanzara, G. Schmidt, Approximate Approximations: recent developments in the computation of high dimensional potentials, to appear in St. Petersburg Math. J.

\bibitem{MSbook} V. Maz'ya, G. Schmidt, Approximate Approximations, AMS  (2007).
%
\bibitem{RR} R. Rannacher, Finite Element Methods for the Incompressible Navier-Stokes Equations. In: Galdi G.P., Heywood J.G., Rannacher R. (eds) Fundamental Directions in Mathematical Fluid Mechanics. Advances in Mathematical Fluid Mechanics. Birkh\"auser, Basel 2000.

\bibitem{GS} G. Seregin, Lecture notes on regularity theory for the Navier-Stokes equations, World Scientific 2014.

\bibitem{So} V. A. Solonnikov , On estimates of solutions of the non-stationary Stokes problem in anisotropic Sobolev spaces and on estimates for the resolvent of the Stokes operator, Russ. Math. Surv. 58, 331-365 (2003).

%
\bibitem{Mo} H.~Takahasi and M.~Mori, Doubly exponential formulas for numerical integration,
Publ. RIMS, Kyoto Univ. 9, 721-741(1974).

\bibitem{RT} R. Teman, Navier-Stokes equations: theory and numerical analysis, North Holland Publ. Company 1977.


\bibitem{TG} E.A. Thomann, R.B. Guenther, The Fundamental Solution of the Linear Navier-Stokes Equations for Spinning bodies in Three Spatial Dimensions - Time Dependent Case, J. Math. Fluid Mech. 8 , 77-98 (2006).

\end{thebibliography}
\end{document}